\documentclass[a4paper]{article}

\usepackage{fouriernc}

\usepackage{amsmath,amsthm,amssymb,latexsym}

\newtheorem{theorem}{Theorem}[section]
\newtheorem{proposition}[theorem]{Proposition}
\newtheorem{lemma}[theorem]{Lemma}

\newtheorem{remark}[theorem]{Remark}

\newtheorem{assumption}[theorem]{Assumption}

\newtheorem{claim}[theorem]{Claim}

\begin{document}

\title{On ergodicity of the SAGA-LD algorithm\thanks{The author gratefully 
acknowledges the support of 
the National Research, Development and Innovation Office (NKFIH) through grants K 143529
and KKP 137490.}}

\author{Mikl\'os R\'asonyi\thanks{HUN-REN Alfr\'ed R\'enyi 
Institute of Mathematics and E\"otv\"os Lor\'and University, Budapest}}

\date{\today}

\maketitle

\begin{abstract}
The so-called SAGA-LD algorithm is used for efficient sampling from high-dimensional distributions in machine learning.
Its intricate dynamics resists standard approaches of Markov chain theory. We prove,
using a model-specific method, that SAGA-LD converges to a limiting
distribution and a law of large numbers holds. This provides a reassuring theoretical guarantee for
using SAGA-LD in practice for calculating various functionals of the target distribution.
\end{abstract}

\noindent\textbf{Keywords.} sampling, variance reduction, stochastic gradient Langevin dynamics,
SAGA algorithm, law of large numbers

\section{Introduction}\label{intro}

Non-convex optimization procedures in machine learning are often based on the Langevin equation
\begin{equation}\label{lange}
dL_{t}=-F(L_{t})\, dt +\sqrt{2}dB_{t},
\end{equation}
where $B_{t}$ is some high-dimensional standard Brownian motion and $F=\nabla U$ for some
function $U\geq 0$. Under mild conditions, the process $L_{t}$ has an invariant density proportional
to $\exp(-U)$. A simple way of (approximate) sampling from this density would be Euler-Maruyama approximations
of \eqref{lange}. In practice, however, $F$ is not available or prohibitively expensive to evaluate.

We consider the case where $F(x)$ is of the form $\frac{1}{N}\sum_{i=1}^{N}F_{i}(x)$ for some functions
$F_{i}:\mathbb{R}^{d}\to\mathbb{R}^{d}$, $i=1,\ldots, N$.
An often used sampling algorithm was first proposed in \cite{teh}: Stochastic Gradient Langevin Dynamics (SGLD), 
initialized from  some $\tilde{X}_{0}\in\mathbb{R}^{d}$, and given by the recursion
\begin{equation}\label{sgld}
\tilde{X}_{n+1}=\tilde{X}_{n}-\eta F_{S_{n}}(\tilde{X}_{n})+\sqrt{2\eta}\xi_{n+1},	
\end{equation}
where $\xi_{n}$ are an i.i.d. sequence of $d$-dimensional standard Gaussian variables, $\eta>0$ is the step size, and
$S_{n}$ are a sequence of i.i.d. uniform draws from the set $\{1,2,\ldots,N\}$.
In practice, i.i.d. \emph{minibatches} are drawn from the $F_{i}$ and their average is used in \eqref{sgld} instead of just
$F_{S_{n}}(X_{n})$. For the present, theoretical study we contend with this simplified version of SGLD. 

Variance-reduction techniques produced several variants of SGLD, most notably SAGA-LD and SVRG-LD,
see e.g.~\cite{chatterji}. They show superior performance at the price of a more involved
dynamics. In the present paper we concentrate on the SAGA-LD algorithm, to be defined in Section \ref{2} below.


That algorithm is mathematically challenging. SAGA-LD, as a Markov chain, does not seem to satisfy 
either of the two standard hypotheses of 
stochastic stability theory for general state space Markov
chains: Foster-Lyapunov-type drift condition 
and minorization on a small set, see e.g.\ \cite{mt}.
As these two standard criteria fail,
we take an unusual route in this paper, originating
from \cite{eureka}, that establishes stochastic stability and the convergence of
ergodic averages via showing a strong mixing property. In doing so we do not (and cannot)
rely on, say, Harris recurrence, but on the construction of specific couplings based on the 
dynamics of SAGA-LD.

In Section \ref{2} we formulate precisely our hypotheses and main result Theorem \ref{main}:
ergodic averages of functionals of the SAGA-LD iterates converge in probability.
The latter result provides a reassuring theoretical guarantee 
for numerical work with SAGA-LD.
Remark \ref{conclusions} discusses
possible future directions. Section \ref{coupling}
contains the proofs. Section \ref{alpha} presents background material about strongly mixing processes.

\section{Definitions and results}\label{2}

Fix our probability space $(\Omega,\mathcal{F},P)$, $E[\cdot]$ denotes the corresponding expectation. 
Let an i.i.d.\ sequence $\xi_{n}$, $n\in\mathbb{N}$ of $d$-dimensional standard Gaussian random variables be given.
Furthermore, let a sequence of i.i.d.\ random variables $S_{n}$, $n\in\mathbb{N}$ be drawn uniformly
from the set $\mathcal{N}:=\{1,2,\ldots,N\}$. The initial state of the sampling algorithm is a random variable (or just
a constant) $X_{0}\in\mathbb{R}^{d}$. 
We assume that
$X_{0}$, $(\xi_{n})_{n\in\mathbb{N}}$, $(S_{n})_{n\in\mathbb{N}}$ are independent. Scalar product in $\mathbb{R}^{d}$
is denoted $\langle\cdot,\cdot\rangle$; $|\cdot|$ is the Euclidean norm in $\mathbb{R}^{m}$, where $m$ is always clear from
the context. $\mathcal{B}(\mathbb{R}^{m})$ are the Borel sets of $\mathbb{R}^{m}$. 

The ``Stochastic Average Gradient Langevin
Dynamics'' (SAGA-LD) is then defined as follows: there is an $N\times d$-dimensional process $G_{n}$, $n\in\mathbb{N}$ which stores
the current estimates for the respective functions 
$F_{i}$, $i=1,\ldots,N$. Indeed, initialize this process by $G_{0}^{i}:=F_{i}(X_{0})$, $i=1,\ldots,N$ and set
\begin{equation}\label{gradient}
G^{i}_{n+1}=1_{\{S_{n}\neq i\}}G^{i}_{n}+1_{\{S_{n}=i\}}F_{i}(X_{n}),\ n\in\mathbb{N}.	
\end{equation}
The $d$-dimensional state variable for sampling is defined by the recursion
\begin{equation}\label{state}
X_{n+1}=X_{n}-\frac{\eta}{N}\sum_{i=1}^{N}G^{i}_{n}-
\eta(F_{S_{n}}(X_{n})-G^{S_{n}}_{n})+\sqrt{2\eta}\xi_{n+1},\ n\in\mathbb{N}.	
\end{equation}

Here the updates 
\begin{equation}\label{tilde}
\frac{\eta}{N}\sum_{i=1}^{N}G^{i}_{n}+
\eta(F_{S_{n}}(X_{n})-G^{S_{n}}_{n}) 
\end{equation}
replace the term $\eta F_{S_{n}}(\tilde{X}_{n})$ appearing in the SGLD algorithm \eqref{sgld} above.

Our assumptions are listed below. They essentially coincide with those of \cite{difan}
and are weaker than those of \cite{chatterji} (they do not imply
convexity of the function $U$ in Section \ref{intro}).

\begin{assumption}\label{init} 
$$
E[|X_{0}|^{2}]<\infty.
$$
\end{assumption}

\begin{assumption}\label{dissip}
There exist $c_{1},c_{2}>0$ such that
$$
\left\langle \frac{1}{N}\sum_{i=1}^{N}F_{i}(x),x\right\rangle\geq c_{2}|x|^{2}-c_{1},\ x\in\mathbb{R}^{d}.
$$
We may and will assume $c_{2}\leq 1$.
\end{assumption}

\begin{assumption}\label{lip} For each $1\leq i\leq N$, $F_{i}$ is Lipschitz continuous, 
that is, there is $M>0$ such that
$$
|F_{i}(x)-F_{i}(y)|\leq M|x-y|
$$ 
for each $i$, and for all $x,y\in \mathbb{R}^{d}$. We may and will assume $M\geq 1$.
\end{assumption}

The total variation distance of two probabilities $\mu_{1},\mu_{2}$ on $\mathcal{B}(\mathbb{R}^{d})$ is
defined as
$$
d_{TV}(\mu_{1},\mu_{2}):=\inf_{X_{1}\sim \mu_{1},X_{2}\sim \mu_{2}}P(X_{1}\neq X_{2}),
$$
here the infimum ranges over all random variables $X_{1}$ (resp. $X_{2}$) with law $\mu_{1}$ (resp. $\mu_{2}$).
The main achievement of the current paper is stated now. 

\begin{theorem}\label{main} Let Assumptions \ref{init}, \ref{dissip} and \ref{lip} hold. There is $\eta_{0}>0$ such that,
for $0<\eta\leq \eta_{0}$ 
there exists a probability $\mu_{*}=\mu_{*}(\eta)$ on $\mathbb{R}^{d}\times\mathbb{R}^{Nd}$ such that 
$$
d_{TV}(\mathrm{Law}(X_{n},G_{n}),\mu_{*})\to 0,\ n\to\infty.
$$	
For all measurable 
functions $\phi:\mathbb{R}^{(N+1)d}\to\mathbb{R}$ such that 
\begin{equation}\label{cillag}
|\phi(u)|\leq C(1+|u|)	
\end{equation} 
with some $C>0$,
\begin{equation}\label{lln}
E\left[\left|\frac{\phi(X_{1},G_{1})+\ldots+\phi(X_{n},G_{n})}{n}-\int_{\mathbb{R}^{(N+1)d}}\phi(u)\mu_{*}(du)\right|\right]\to 0,
\ n\to\infty.
\end{equation}
A fortiori, the averages converge in probability. 
\end{theorem}
 
\begin{remark}\label{discussions}
{\rm Let the first marginal of $\mu_{*}(\eta)$ be denoted by $\upsilon_{*}(\eta)$,{}
this is a probability on $\mathcal{B}(\mathbb{R}^{d})$. 
It can be shown under suitable assumptions, see e.g.\ \cite{chatterji}, that $\upsilon_{*}(\eta)$ is close to
the invariant law $\upsilon$ of \eqref{lange}, for $\eta$ small enough. Then, for large $n$, the law of $X_{n}$
approximates $\upsilon_{*}(\eta)$, hence also $\upsilon$. In other words, $X_{n}$ approximately samples
from the target law $\upsilon$. 

In practice, however, what we need is approximating quantities like $\Phi:=\int_{\mathbb{R}^{d}}\phi(x)\upsilon(dx)$ for
some function $\phi$, using only \emph{one generated realization} of $X_{n}$, $n\in\mathbb{N}$. In other words,
we need that averages like
$$
\frac{\phi(X_{m+1})+\ldots+\phi(X_{m+n})}{n}
$$
are close to $\Phi^{*}(\eta):=\int_{\mathbb{R}^{d}}\phi(x)\upsilon_{*}(\eta)(dx)$ for $n$ large. ($\Phi^{*}(\eta)$ will
be close to $\Phi$ for $\eta$ small enough since $\upsilon_{*}(\eta)$ is close to $\upsilon$.) 

To sum up, in order to use SAGA-LD in numerical work
with confidence, a law of large numbers should be established. 
This is precisely the content of \eqref{lln} in
Theorem \ref{main}. No such results have been available
in the literature so far. 
}	
\end{remark}

\begin{remark}\label{conclusions} {\rm 
Determining the convergence rate of $(X_{n},G_{n})$ to $\mu_{*}$ is an open problem of great interest. Our present method
achieves a poor rate only, see Remark \ref{poor} below. Note, however, that
the dynamics of SAGA-LD resists the textbook approach (Foster-Lyapunov plus minorization on small sets) so even 
obtaining convergence is challenging.

It would also be desirable
to have almost sure convergence in \eqref{lln}. It seems that other variance-reduced sampling schemes could be treated along the same lines,
in particular SVRG-LD. These require, however, different state-space representations hence they cannot be presented in one
general theorem and the arguments of the present paper should be adapted one-by-one to them. For this reason we
stick to SAGA-LD in the present note, which seems to be the most difficult to handle.}
\end{remark}

\section{Coupling construction}\label{coupling}

The indicator function
of an event $A\in\mathcal{F}$ is denoted $1_{A}$.
Define the filtration 
$$
\mathcal{H}_{n}:=\sigma(X_{0},\xi_{j},\ 0\leq j\leq n,S_{j}, 0\leq j<n),\ n\in\mathbb{N}.
$$
Clearly, the process $(X_{n})_{n\in\mathbb{N}}$ is adapted to $(\mathcal{H}_{n})_{n\in\mathbb{N}}$ and 
$S_{n}$ is independent of $\mathcal{H}_{n}$, for all $n\geq 0$. 
Assumptions \ref{init}, \ref{dissip} and \ref{lip} remain in force in the rest of the paper.


\subsection{Moment estimates}

Denote $\hat{M}:=\max_{i=1,\ldots,N}|F_{i}(0)|$.
It follows from Assumption \ref{lip} that, for all $i$, $x$, 
\begin{equation}\label{muc}
|F_{i}(x)|\leq \hat{M}+M|x|.
\end{equation}

A key initial observation is that the process $(X_{n},G_{n})$, $n\in\mathbb{N}$ remains bounded in $L^{2}$
provided that $(X_{0},G_{0})$ are in $L^{2}$.
This was already noticed in Lemma C.7 of \cite{difan} but only in the case $X_{0}=0$. In our arguments below
we will need the case with a general initialization.

\begin{lemma}\label{l2} If $\eta$ is small enough then
$$
\sup_{n\in\mathbb{N}}E[|X_{n}|^{2}+|G_{n}|^{2}]<\infty.
$$		
More precisely,
\begin{eqnarray}\nonumber
& & \sup_{n\in\mathbb{N}}E\left[|X_{n}|^{2}+\sum_{i=1}^{N}|G_{n}^{i}|^{2}\right]\leq\\
\hat{C} &:=&
2(N+1)[\hat{M}^{2}+M^{2}]\frac{2(2d+c_{1}+2\hat{M}^{2}+E[|X_{0}|^{2}])}{c_{2}\eta}
+\sum_{i=1}^{N}E\left[|G_{0}^{i}|^{2}\right],\label{chat}
\end{eqnarray}
whenever 
\begin{equation}\label{etta}
\eta\leq \frac{c_{2}}{8M^{2}}.	
\end{equation}
\end{lemma}
\begin{proof}{Proof} Let \eqref{etta} hold. 
We set $L_{n}:=\sup_{0\leq l\leq n}E[|X_{l}|^{2}]$.  We will prove by induction that 
\begin{equation}\label{indu}
L_{n}\leq \frac{2(2d+c_{1}+2\hat{M}^{2}+E[|X_{0}|^{2}])}{c_{2}\eta}.
\end{equation}
and also 
\begin{equation}\label{indu2}
E[|G_{n}^{i}|^{2}]\leq E[|G_{0}^{i}|^{2}]+2[\hat{M}^{2}+M^{2}L_{n}],\ i=1,\ldots,N,
\end{equation}
for all $n$. \eqref{indu} and \eqref{indu2} are trivial for $n=0$.
So we assume that \eqref{indu}, \eqref{indu2} are true for $n$ and we will verify them for $n+1$.

The events $A_{n-1}:=\{S_{n}=S_{n-1}\}$, $A_{n-2}:=\{S_{n}\neq S_{n-1},\ S_{n}=S_{n-2}\}$, \ldots,{}
$A_{0}:=\{S_{n}\neq S_{n-1},\ldots,S_{n}\neq S_{1}\}$ form a partition of $\Omega$.{}
It is clear that $G_{n}^{S_{n}}=F_{S_{n}}(X_{j})$ on $A_{j}$, $j=0,\ldots,n-1$.{}

Since $\xi_{n+1}$ has mean $0$ and it is independent of $\sigma(\mathcal{H}_{n},S_{n})$, we may write
\begin{eqnarray*}
E[|X_{n+1}|^{2}]= E\left[\left|X_{n}-\frac{\eta}{N}\sum_{i=1}^{N}G^{i}_{n}-
\eta(F_{S_{n}}(X_{n})-G^{S_{n}}_{n})\right|^{2}\right]+2\eta E[|\xi_{n+1}|^{2}].
\end{eqnarray*}
Furthermore,
\begin{eqnarray}\nonumber
& & E[|X_{n+1}|^{2}]= 
E\left[\left|X_{n}-\frac{\eta}{N}\sum_{i=1}^{N}F_{i}(X_{n})\right|^{2}\right]\\
&+&  E\left[\left|\frac{\eta}{N}\sum_{i=1}^{N}F_{i}(X_{n})-\frac{\eta}{N}\sum_{i=1}^{N}G^{i}_{n}-
\eta(F_{S_{n}}(X_{n})-G^{S_{n}}_{n})\right|^{2}\right]+2\eta E[|\xi_{n+1}|^{2}].\label{uni} 	
\end{eqnarray}
Indeed,
\begin{eqnarray*}
& & E\left[\left\langle X_{n}-\frac{\eta}{N}\sum_{i=1}^{N}F_{i}(X_{n}),
\frac{1}{N}\sum_{i=1}^{N}F_{i}(X_{n})-\frac{1}{N}\sum_{i=1}^{N}G^{i}_{n}-
(F_{S_{n}}(X_{n})-G^{S_{n}}_{n})\right\rangle\right]\\
&=& E\left[E\left[\left\langle X_{n}-\frac{\eta}{N}\sum_{i=1}^{N}F_{i}(X_{n}),
\frac{1}{N}\sum_{i=1}^{N}F_{i}(X_{n})-\frac{1}{N}\sum_{i=1}^{N}G^{i}_{n}-
(F_{S_{n}}(X_{n})-G^{S_{n}}_{n})\right\rangle\Big|\mathcal{H}_{n}\right]\right]\\
&=& E\left[\left\langle X_{n}-\frac{\eta}{N}\sum_{i=1}^{N}F_{i}(X_{n}),
E\left[\frac{1}{N}\sum_{i=1}^{N}F_{i}(X_{n})-\frac{1}{N}\sum_{i=1}^{N}G^{i}_{n}-
(F_{S_{n}}(X_{n})-G^{S_{n}}_{n})\Big|\mathcal{H}_{n}\right]\right\rangle\right]
\end{eqnarray*}
and
$$
E\left[\frac{1}{N}\sum_{i=1}^{N}F_{i}(X_{n})-\frac{1}{N}\sum_{i=1}^{N}G^{i}_{n}-
(F_{S_{n}}(X_{n})-G^{S_{n}}_{n})\Big|\mathcal{H}_{n}\right]=0.
$$
The second term on the right-hand side of \eqref{uni} can be estimated as follows:
\begin{eqnarray*}
& & E\left[\left|\frac{1}{N}\sum_{i=1}^{N}F_{i}(X_{n})-\frac{1}{N}\sum_{i=1}^{N}G_{n}^{i}
-F_{S_{n}}(X_{n})+G_{n}^{S_{n}}\right|^{2}\right]\\
&=& E\left[ E\left[\left|\frac{1}{N}\sum_{i=1}^{N}F_{i}(X_{n})-\frac{1}{N}\sum_{i=1}^{N}G_{n}^{i}
-F_{S_{n}}(X_{n})+G_{n}^{S_{n}}\right|^{2}\Big|\mathcal{H}_{n}\right]\right]\\
&=& E\left[ E\left[\left|F_{S_{n}}(X_{n})-G_{n}^{S_{n}}-E\left[F_{S_{n}}(X_{n})-G_{n}^{S_{n}}
\Big|\mathcal{H}_{n}\right]\right|^{2}\Big|\mathcal{H}_{n}\right]\right]\\
&\leq& E\left[ E\left[\left|F_{S_{n}}(X_{n})-G_{n}^{S_{n}}\right|^{2}\Big|\mathcal{H}_{n}\right]\right]\\
&=& \sum_{j=0}^{n-1}E\left[\left|F_{S_{n}}(X_{n})-F_{S_{n}}(X_{j})\right|^{2} 1_{A_{j}}\right]\\
&\leq& M^{2}\sum_{j=0}^{n-1}E\left[\left|X_{n}-X_{j}\right|^{2} 1_{A_{j}}\right]\\
&\leq& 2M^{2}E[|X_{n}|^{2}]+2M^{2}\sum_{j=0}^{n-1}E\left[\left|X_{j}\right|^{2} 1_{A_{j}}\right]\\
&\leq& 2M^{2}E[|X_{n}|^{2}]+2M^{2}\sum_{j=0}^{n-1}E\left[\left|X_{j}\right|^{2}\right]P(A_{j})\\
&\leq& 4M^{2}\max_{0\leq j\leq n} E[|X_{j}|^{2}]=4M^{2}L_{n},
\end{eqnarray*}	
where we used that, for any integrable random variable $Y$, 
$$
E[|Y-E[Y\vert\mathcal{H}_{n}]|^{2}\vert\mathcal{H}_{n}]\leq E[|Y|^{2}\vert\mathcal{H}_{n}]
$$
and that $A_{j}$ is independent of $X_{j}$ (integrability is clear from the induction hypothesis).
Since $\eta\leq c_{2}/(8M^{2})\leq c_{2}/(2M^{2})$, Assumption \ref{dissip} implies
\begin{eqnarray*}
& & E\left[\left|X_{n}-\frac{\eta}{N}\sum_{i=1}^{N}F_{i}(X_{n})\right|^{2}\right]\leq E[|X_{n}|^{2}]+\frac{\eta^{2}}{N^{2}}
E\left[\left|\sum_{i=1}^{N}F_{i}(X_{n})\right|^{2}\right]-\\
&-& \frac{2\eta}{N}E\left[\langle X_{n},\sum_{i=1}^{N}F_{i}(X_{n})\rangle\right]\leq c_{1}+(1-2c_{2}\eta) E[|X_{n}|^{2}]
+\frac{\eta^{2}}{N^{2}}E[ (N(\hat{M}+M|X_{n}|))^{2}]\\
&\leq & c_{1}+(1-2c_{2}\eta)E[|X_{n}|^{2}]+2\eta^{2}\hat{M}^{2}+2\eta^{2}M^{2}E[|X_{n}|^{2}]\\
&\leq&	c_{1}+(1-c_{2}\eta)E[|X_{n}|^{2}]+2\eta^{2}\hat{M}^{2}.
\end{eqnarray*} 
Inserting our estimations into \eqref{uni},
\begin{eqnarray}\nonumber
E[|X_{n+1}|^{2}]&\leq& c_{1}+(1-c_{2}\eta)E[|X_{n}|^{2}]+2\eta^{2}\hat{M}^{2}+\eta^{2}4M^{2}L_{n}+2\eta d\\
&\leq& {}
c_{1}+(1-c_{2}\eta/2)L_{n}+2\hat{M}^{2}+2d,\label{mal}
\end{eqnarray}
since $\eta\leq c_{2}/(8M^{2})\leq 1$ by Assumptions \ref{dissip} and \ref{lip}.

Since \eqref{indu} holds for $n$, we see that
\begin{eqnarray*}
E[|X_{n+1}|^{2}]&\leq& 
c_{1}+2d+2\hat{M}^{2}+(1-c_{2}\eta/2)\frac{2(2d+c_{1}+2\hat{M}^{2}+E[|X_{0}|^{2}])}{c_{2}\eta}\\
&=& \frac{2(2d+c_{1}+2\hat{M}^{2}+E[|X_{0}|^{2}])}{c_{2}\eta}-E[|X_{0}|^{2}]\\
&\leq& \frac{2(2d+c_{1}+2\hat{M}^{2}+E[|X_{0}|^{2}])}{c_{2}\eta},	
\end{eqnarray*}
so \eqref{indu} holds for $n+1$.

For each $i=1,\ldots,N$, we have by \eqref{muc} that 
$$
|G_{n+1}^{i}|\leq \sum_{j=0}^{n-1}1_{A_{j}}|F_{i}(X_{j})|\leq \sum_{j=0}^{n-1}1_{A_{j}}(\hat{M}+M|X_{j}|),
$$
so, by independence of $X_{j}$ and $A_{j}$,
\begin{eqnarray*}\label{gis}
E[|G_{n+1}^{i}|^{2}] &\leq& \sum_{j=0}^{n-1}E[1_{A_{j}}(\hat{M}+M|X_{j}|)^{2}]=\\
&=& \sum_{j=0}^{n-1}P(A_{j})E[(\hat{M}+M|X_{j}|)^{2}]\leq 2\sum_{j=0}^{n-1}P(A_{j})\left[\hat{M}^{2}+M^{2}E[|X_{j}|^{2}]\right]\leq \\
&\leq& 2[\hat{M}^{2}+M^{2}L_{n}]. 
\end{eqnarray*}
This completes the induction step.\hfill $\Box$
\end{proof}


\subsection{Representation by random maps}

Let an $\eta$ satisfying \eqref{etta} be fixed for the rest of the paper. 
We start by remarking that, although the process $X$ badly fails the Markov property, 
the pair $(X,G)$ \emph{is} a general state space
Markov chain.

Representing Markov chains by random maps is a well-known approach to the analysis of
their ergodic properties, see e.g.\ 
\cite{majumdar,bw,bwbook}. The choice of these maps is crucial and allows for
the realization of various couplings.

Let us fix $0<\varepsilon\leq 1$. We will define a specific representation for $(X_{n},G_{n})$, $n\in\mathbb{N}$ with the help
of random mappings
such that, starting the chain from two different initializations $(X_{0},G_{0})$ and
$({X}_{0}',{G}_{0}')$, for $n$ large enough, $(X_{n},G_{n})=({X}_{n}',{G}_{n}')$ will
hold outside an event of probability at most $3\varepsilon$, see Proposition \ref{heart} below
for the exact formulation.

It will easily follow from this construction that $\mathrm{Law}(X_{n},G_{n})$ converges to a limiting
law as $n\to\infty$ and that the process $(X_{n},G_{n})$, $n\in\mathbb{N}$ is strongly mixing: see Section
\ref{alpha} for definitions and Subsection \ref{last} for details.

We may and will assume that there is a sequence $(U_{n},U_{n}')$, $n\in\mathbb{N}$ of
i.i.d.\ uniform random variables on $[0,1]^{2}$ which is independent of $X_{0}$, $(S_{n})_{n\in\mathbb{N}}$.
We will construct a process using random function iterations driven by the noise
sequence $(U_{n},U_{n}')$ that has 
the same law as $(X_{n},G_{n})$, $n\in\mathbb{N}$.

The following result is well-known, see e.g.\ page 228 of \cite{bwbook} (or Lemma A.5 of \cite{laurence}
for a similar statement).

\begin{lemma}\label{construct}
Let $\kappa(x,A)$ be a probabilistic kernel, $x\in\mathbb{R}^{m_{1}},A\in\mathcal{B}(\mathbb{R}^{m_{2}})$. 	
Let $U$ be uniform on $[0,1]$. Then there is a function $\chi:\mathbb{R}^{m_{1}}\times[0,1]\to\mathbb{R}^{m_{2}}$ such that
$\chi(x,U)$ has law $\kappa(x,\cdot)$, for all $x$.\hfill $\Box$
\end{lemma}

\begin{remark}\label{spec}{\rm In particular, it follows from Lemma \ref{construct} that given a probability $\kappa$ on
$\mathcal{B}(\mathbb{R}^{m})$ there is a function $\chi:[0,1]\to\mathbb{R}^{m}$ such that $\chi(U)$ has law $\kappa$.}
\end{remark}

Now we briefly explain the strategy of the couplings that will be worked out in the present subsection.
Using the fact that Gaussian densities are bounded from below on compacts (Claim \ref{lower}),
Lemma \ref{function} constructs mappings that are constant on some compact sets with a positive
probability (namely, whenever the auxiliary uniform random variables $U_{n}'$ are below the level $\beta$).
Using these mappings, a copy of the original process $(X_{n},G_{n})$, started from an arbitrary initialization, is created 
by the iterative definition \eqref{iterate}. Proposition \ref{heart} is the core of our arguments: it establishes that
two such processes eventually coincide with a probability arbitrarily close to $1$. The proof of Proposition \ref{heart}
proceeds in several steps: Claim \ref{markov} establishes that both processes enter a fixed ball around the origin with high probability;
the technical Claim \ref{fontos} defines a ``fortunate'' event where the process stays in a suitable compact and
the couplings are successful in every step of the given block of length $(N+1)$; finally, Claim \ref{pici} estimates
the probability of this fortunate event recursively. After these preparations, one can easily conclude the
proof of Proposition \ref{heart}.


The dynamics \eqref{state} shows that the conditional law of $X_{n+1}$ with respect to $X_{n}=x$, $G_{n}=g$, $S_{n}=s$ 
is $d$-dimensional Gaussian with mean
$$
h(x,g,s):=x-\frac{\eta}{N}\sum_{i=1}^{N}g^{i}-\eta(F_{s}(x)-g^{s}),\ x,g^{i}\in\mathbb{R}^{d},\ i=1,\ldots,N,\ s\in\mathcal{N},
$$
and diagonal covariance matrix with $2\eta$ at every diagonal entry. This observation will be crucial
in the arguments below.

Choose 
\begin{equation}\label{mekka}
K=K(\varepsilon):=\sqrt{\frac{2N+2}{\varepsilon}\hat{C}},
\end{equation}
recall \eqref{chat}.
For $r>0$, $B(r)$ will denote the closed ball of radius $r$ around the origin of $\mathbb{R}^{d}$.
Notice that $\hat{C}\geq 1$ hence also $K\geq 1$.

\begin{claim}\label{lower} There is $\beta=\beta(\varepsilon)>0$ such that, for $K=K(\varepsilon)$ and
for all $x\in B(3\hat{M}+4MK)$,
$$
P\left(x+\sqrt{2\eta}\xi_{0}\in B(1)\cap A\right)\geq \beta \nu(A),
$$
where $\nu$ is the uniform law on $B(1)$.
\end{claim}
\begin{proof}{Proof.}
Indeed,	
\begin{eqnarray*}
& & P\left(x+\sqrt{2\eta}\xi_{0}\in B(1)\cap A\right)\geq \int_{B(1)\cap A}\frac{\exp(-|u-x|^{2}/4\eta)}{(\sqrt{2\eta})^{d}(\sqrt{2\pi})^{d}}\, du\\
&\geq& 	\frac{\exp(-(3\hat{M}+4MK+1)^{2}/4\eta)}{(\sqrt{2\eta})^{d}(\sqrt{2\pi})^{d}}\mathrm{Leb}(B(1)\cap A)=\mathrm{Leb}(B(1))
\frac{\exp(-(3\hat{M}+4MK+1)^{2}/4\eta)}{(\sqrt{2\eta})^{d}(\sqrt{2\pi})^{d}}\nu(A),
\end{eqnarray*}
so we define 
\begin{equation}\label{mokka}
\beta:=\mathrm{Leb}(B(1))
\frac{\exp(-(3\hat{M}+4MK)+1)^{2}/4\eta)}{(\sqrt{2\eta})^{d}(\sqrt{2\pi})^{d}}.
\end{equation}
\hfill $\Box$
\end{proof}

Now we construct a suitable (random) function that can be iterated to produce a copy of $(X_{n},G_{n})$, $n\in\mathbb{N}$ but, at the same time,
realize useful couplings.
Define the kernel
$$
q(x,g,s,A):=\int_{A}\frac{\exp(-(u-h(x,g,s))^{2}/4\eta)}{(\sqrt{2\eta})^{d}(\sqrt{2\pi})^{d}}\, du,\ A\in\mathcal{B}(\mathbb{R}^{d}),(x,g)\in\mathbb{R}^{(N+1)d},
s\in\mathcal{N}.
$$
Notice that $q(x,g,s,\cdot)$ is the law of $X_{1}$ conditionally on $X_{0}=x,G_{0}=g,S_{0}=s$.

For a vector $v\in \mathbb{R}^{(N+1)d}$, $v^{0}$ will refer
to its first $d$ coordinates, while $v^{l}$ with $l=1,\ldots,N$ will refer to its $(l+1)$th block of $d$ coordinates.
\begin{lemma}\label{function} 
There is a function $f:\mathbb{R}^{(N+1)d}\times [0,1]^{2}\times \mathcal{N}\to \mathbb{R}^{(N+1)d}$ such that
$$
\mathrm{Law}(f(x,g,U_{0},U_{0}',s))=q(x,g,s,\cdot){}
$$ and, for each $u'\leq \beta$, 
the mapping $(x,g,s)\to f^{0}(x,g,u,u',s)$ is constant on the set
$$
\{(x,g,s):x\in B(K),\ g^{i}\in B(\hat{M}+MK),\ i=1,\ldots,N,\ s\in\mathcal{N}\},$$ for all $u$.	
\end{lemma}
\begin{proof}{Proof.} First, let $\phi:[0,1]\to \mathbb{R}^{d}$ be such that $\phi(U)$ is uniform on $B(1)$, see Lemma
\ref{construct} and Remark \ref{spec}. Let $\theta:[0,1]\to \mathbb{R}^{d}$ be such that
$\theta(U)$ has standard $d$-dimensional Gaussian law. 
By Claim \ref{lower}, 
$$
\kappa(x,g,s,A):=\frac{q(x,g,s,A)-\beta \nu(A)}{1-\beta},\ x\in B(K),g^{1},\ldots,g^{N}\in B(\hat{M}+MK),
A\in\mathcal{B}(\mathbb{R}^{d})
$$
is a probabilistic kernel (i.e.\ the numerator is a non-negative kernel). 
Indeed, if $x\in B(K)$, $g^{i}\in B(\hat{M}+MK)$ for $i=1,\ldots,N$ then
$$
\left|-\frac{1}{N}\sum_{i=1}^{N}g^{i}-(F_{s}(x)-g^{s})\right|\leq \hat{M}+MK+\hat{M}+MK+\hat{M}+MK,
$$
hence, recalling that $\eta\leq 1$ by \eqref{etta} and $M\geq 1$ by Assumption \ref{lip}, 
$x-\frac{\eta}{N}\sum_{i=1}^{N}g^{i}-\eta(F_{s}(x)-g^{s})\in B(3\hat{M}+4MK)$ and Claim \ref{lower}
implies that $q(x,g,s,A)-\beta \nu(A)\geq 0$.
Let $\chi:\mathbb{R}^{(N+1)d}\times [0,1]\times\mathcal{N}\to\mathbb{R}^{d}$
be such that $\chi(x,g,s,U)$ has law $\kappa(x,g,s,\cdot)$.

We define $f$: if $x\in B(K),g^{i}\in B(\hat{M}+MK)$, $i=1,\ldots,N$ and $s\in\mathcal{N}$ then set
\begin{eqnarray*}
f^{0}(x,g,u,u',s) &:=& \phi(u),\mbox{ if }u'\leq \beta,\\
f^{0}(x,g,u,u',s) &:=& \chi(x,g,s,u),\mbox{ if }u'> \beta,\\
f^{l}(x,g,u,u',s) &:=& F_{l}(x)\mbox{ if }l=s,\\
f^{l}(x,g,u,u',s) &:=& g\mbox{ if }l\neq s.
\end{eqnarray*}
Otherwise set
\begin{eqnarray*}
f^{0}(x,g,u,u',s) &:=& x-\frac{\eta}{N}\sum_{j=1}^{N}g^{j}-\eta(F_{s}(x)-g^{s})+\sqrt{2\eta}\theta(u),\\
f^{l}(x,g,u,u',s) &:=& F_{l}(x)\mbox{ if }l=s,\\
f^{l}(x,g,u,u',s) &:=& g\mbox{ if }l\neq s.
\end{eqnarray*}
The statement about constant mappings is trivial from this construction.
It is also clear that $f(x,g,U_{0},U_{0}',s)$ has the right conditional law for all $x,g,s$.\hfill $\Box$
\end{proof}

Define for $n,m\in\mathbb{N}$, $x\in\mathbb{R}^{d}$, $g\in\mathbb{R}^{Nd}$,
\begin{equation}\label{iterate}
Z^{x,g}_{m,n}:=(x,g)\mbox{ if }n\leq m,\ Z^{x,g}_{m,n}:=f(Z^{x,g}_{m,n-1},U_{n},U_{n}',S_{n-1})\mbox{ if }n>m.
\end{equation}

The next result lies at the heart of our proof.

\begin{proposition}\label{heart} Let $l\in\mathbb{N}$ and let $Q_{l},Q_{l}'\in\mathbb{R}^{(N+1)d}$ be random variables that are
independent of $\sigma(U_{i},U_{i}',i\geq l+1,S_{i}, i\geq l)$. Assume that
$$
\max\{\sup_{n\in\mathbb{N}}E[|Z^{Q_{l}}_{l,l+n}|^{2}],\sup_{n\in\mathbb{N}}E[|Z^{Q_{l}'}_{l,l+n}|^{2}]\}\leq \hat{C}.
$$  
Then there is $n_{0}$ such that, for $n\geq n_{0}$, 
$$
P(Z_{l,l+n}^{Q_{l}}\neq Z_{l,l+n}^{Q_{l}'})\leq 
3\varepsilon.
$$	
Here $n_{0}$ does not depend either on $l$ or on the choice of $Q_{l},Q_{l}'$.
\end{proposition}
\begin{proof}{Proof.} 
We may and will assume $l=0$, the case of general $l$ being only notationally different.
To simplify notation, we will write
$Z_{n}:=Z_{0,n}^{Q_{l}}$ and $Z_{n}':=Z_{0,n}^{Q_{l}'}$.


Unlike in Lemma \ref{function}, for $v\in\mathbb{R}^{(N+1)d}$, $v(i)$ refers to the $d$ 
dimensional vector which is the $(i+1)$th $d$-tuple of $v$, for $i=0,\ldots,N$.

\begin{claim}\label{markov}
For all $n\in\mathbb{N}$, each of the probabilities
$$
P(|Z_{n}(i)|>K),P(|Z_{n}'(i)|>K),\ i=0,1,\ldots,N
$$
are $\leq\varepsilon/(2N+2)$. 
\end{claim}
\begin{proof}{Proof.}
Indeed, by Markov's inequality,
$$
P(|Z_{n}(0)|>K)\leq E[|Z_{n}|^{2}]/K^{2}\leq \frac{\varepsilon}{2N+2}\frac{\hat{C}}{\hat{C}},
$$	
and similarly for each $Z_{n}(i),Z_{n}'(i)$. \hfill $\Box$
\end{proof}

We will consider blocks of length $N+1$. 
Define, for $k\in\mathbb{N}$,
$$
H_{k}:=\{Z_{k(N+1)}=Z_{k(N+1)}'\},\quad p_{k}:=P(H_{k})
$$
and
$$
D_{k}:=\{Z_{k(N+1)}(i)\in B(K), Z_{k(N+1)}'(i)\in B(K),\ i=0,\ldots, N\}.
$$
Recall that $\bar{H}_{k}$ denotes the complement of $H_{k}$ in $\Omega$.
{}
The next observation is crucial: on $\bar{H}_{k}$ (where couplings of $Z,Z'$ have not yet been
realized), if all components of the processes $Z,Z'$ are in $B(K)$ at time $k(N+1)$ (that is on the event $D_{k}$); 
all the coupling steps are successful (that is $U_{j}'\leq \beta$ for all $j$, even $U_{(k+1)(N+1)}'\leq \beta$);{}
and at each step the coordinates storing derivative estimates are also coupled (this is guaranteed by $S_{j}=j-k(N+1)$) then all the
coordinates of $Z,Z'$ will be coupled by the time $(k+1)(N+1)$. This is formalized in the next claim. 

\begin{claim}\label{fontos}
On the event 
\begin{eqnarray*}
I_{k} &:=&
\bar{H}_{k}\cap D_{k}\cap \{U_{j}'\leq \beta,S_{j}=j-k(N+1),\\
& & \forall j=k(N+1)+1,\ldots,(k+1)(N+1)-1, U_{(k+1)(N+1)}'\leq \beta\},	
\end{eqnarray*}
one has $Z_{(k+1)(N+1)}=Z_{(k+1)(N+1)}'$. 	
\end{claim}
\begin{proof}{Proof.} We proceed step-by-step. Recall that $K\geq 1$.
On $\bar{H}_{k}\cap D_{k}\cap \{U_{k(N+1)+1}'\leq \beta\}$ one has 
$Z_{k(N+1)+1}(0)=Z_{k(N+1)+1}'(0)\in B(1)\subset B(K)$, by the construction of $f$.
Also, $Z_{k(N+1)+1}(i),Z_{k(N+1)+1}'(i)\in B(\hat{M}+MK)$ for $i=1,\ldots,N$.{}
Indeed, either $Z_{k(N+1)+1}(i)=Z_{k(N+1)}(i)\in B(K)$ or 
$Z_{k(N+1)+1}(i)=F_{t}(Z_{k(N+1)}(0))$ for some $t\in\mathcal{N}$,
which lies in $B(\hat{M}+MK)$ by \eqref{muc}. Similarly for $Z_{k(N+1)+1}'$.

Next, on $\bar{H}_{k}\cap D_{k}\cap \{U_{k(N+1)+1}'\leq \beta,S_{k(N+1)+1}=1, U_{k(N+1)+2}'\leq \beta\}$
we do not only have $Z_{k(N+1)+2}(0)=Z_{k(N+1)+2}'(0)\in B(1)$ but also
$$
Z_{k(N+1)+2}(1)=F_{1}(Z_{k(N+1)+1}(0))=Z_{k(N+1)+2}'(1)=F_{1}(Z_{k(N+1)+1}'(0))\in B(\hat{M}+M)
$$ and still
$$
Z_{k(N+1)+2}(i)=Z_{k(N+1)+1}(i),Z_{k(N+1)+2}'(i)=Z_{k(N+1)+1}'(i)\in B(\hat{M}+MK)
$$ for $i=2,\ldots,N$.

Carrying on in a similar way, on $I_{k}$ we have
$Z_{(k+1)(N+1)}(0)=Z_{(k+1)(N+1)}'(0)\in B(1)$ and also
$Z_{(k+1)(N+1)}(i)=Z_{(k+1)(N+1)}(i)'\in B(\hat{M}+M)$ for $i=1,\ldots,N${}
and the claim is verified.\hfill $\Box$
\end{proof}

This leads to the following observation.

\begin{claim}\label{pici}
\begin{equation}\label{recursion}
p_{k+1}\geq p_{k}+(1-p_{k}-2\varepsilon)(\beta/N)^{N}\beta 	
\end{equation}
\end{claim}	
\begin{proof}{Proof.}
Notice that $H_{k}\subset H_{k+1}$ by the fact that the construction of $Z$ uses function iterations. 
By Claim \ref{markov}, the event $D_{k}$ has probability at least $1-2\varepsilon$. This means
that $P(\bar{H}_{k}\cap D_{k})\geq 1-p_{k}-2\varepsilon$ and this event is independent of 
$$
\{U_{j}'\leq \beta,S_{j}=j-k(N+1), j=k(N+1)+1,\ldots,(k+1)(N+1)-1,U_{(k+1)(N+1)}\leq\beta\}
$$ 
which has probability
$(\beta/N)^{N}\beta$. From Claim \ref{fontos}, 
$$
p_{k+1}\geq P(H_{k})+P(I_{k})\geq p_{k}+(1-p_{k}-2\varepsilon)(\beta/N)^{N}\beta.{}
$$ 
\hfill $\Box$
\end{proof}

Now we can finish the proof of Proposition \ref{heart}. 
Iterating the inequality \eqref{recursion}, $p_{0}\geq 0$ implies
$$
p_{k}\geq (1-2\varepsilon)[1-(1-\beta(\beta/N)^{N})^{k}],
$$
hence 
\begin{equation}\label{balaton}
P(Z_{k(N+1)}\neq Z_{k(N+1)}')\leq 2\varepsilon + (1-\beta(\beta/N)^{N})^{k},
\end{equation}
which, for $k=k(\varepsilon)$ large enough, is smaller than $3\varepsilon$.
{}
Let $n_{0}:=k(\varepsilon)(N+1)$. If $n\geq n_{0}$ then, by construction, 
$$
P(Z_{n}\neq Z_{n}')\leq P(Z_{k(\varepsilon)(N+1)}\neq Z_{k(\varepsilon)(N+1)}')
\leq 3\varepsilon.
$$
\hfill $\Box$
\end{proof}

\begin{remark}\label{poor}{\rm It is clear that Proposition \ref{heart} continues to hold true (with different $n_{0}$)
even if $\hat{C}$ is replaced by any other positive constant. 

From the arguments of Proposition \ref{heart} above we may read out
a certain convergence rate estimate for the system $(X_{n},G_{n})$, however poor it is. \eqref{mekka} and
\eqref{mokka} imply that $\beta$ and hence also $\beta(\beta/N)^{N}$ can be chosen to be of the form $ce^{-c'/\varepsilon}$ for constants
$c,c'>0$. Let us choose $Z_{0}'$ with the law $\mu_{*}$ given by Theorem \ref{main}, and let us choose $Z_{0}$
with the same law as $(X_{0},G_{0})$.
Then \eqref{balaton} implies that $d_{TV}(\mathrm{Law}(X_{k(N+1)},G_{k(N+1)}),\mu_{*})\leq 3\varepsilon$ whenever
$(1-ce^{-c'/\varepsilon})^{k}\leq \varepsilon$ which holds provided that
$k\geq c''|\ln(\varepsilon)|e^{c'/\varepsilon}$ for some constant $c''$. This means that the convergence
speed estimate for $d_{TV}(\mathrm{Law}(X_{k(N+1)},G_{k(N+1)}),\mu_{*})$ is of the order $O(1/\ln(k))$.
}
\end{remark}

\subsection{Last touches}\label{last}

\begin{claim}\label{caim} The sequence $\mathrm{Law}(X_{n},G_{n})$, $n\in\mathbb{N}$ is Cauchy for the metric $d_{TV}$. 
\end{claim}
\begin{proof}{Proof.}
Let $n>m\geq 0$ be given. By construction, the law of $(X_{m},G_{m})$ equals that of $Z^{X_{0},G_{0}}_{0,m}$
and the law of $(X_{n},G_{n})$ equals that of $Z^{\hat{X}_{n-m},\hat{G}_{n-m}}_{0,m}$ where $(\hat{X}_{n-m},\hat{G}_{n-m})$ has the same law
as $(X_{n-m},G_{n-m})$ but it is independent of $\mathcal{H}_{\infty}\vee \sigma(U_{i},U_{i}',i\geq 1)$. Notice that, by Lemma \ref{l2}, Proposition \ref{heart} applies and
$P(Z^{X_{0},G_{0}}_{0,m}\neq Z^{\hat{X}_{n-m},\hat{G}_{n-m}}_{0,m})$ is arbitrarily small for $m$ large enough, which then also shows
that $d_{TV}(\mathrm{Law}(X_{n},G_{n}),\mathrm{Law}(X_{m},G_{m}))$ is arbitrarily small for $m$ large enough.
\hfill $\Box$	
\end{proof}

\begin{claim}\label{mixxi}
The process $(X_{n},G_{n})$ is strongly mixing, that is, $\alpha_{(X,G)}(n)\to 0$ as $n\to\infty$.	
\end{claim}
\begin{proof}{Proof.}
Set $x_{0}=0$, $g_{0}^{i}:=F_{i}(0)$, $i=1,\ldots,N$. Note that $|g_{0}|^{2}\leq N\hat{M}^{2}$. Lemma \ref{l2} implies that 
$$
\sup_{n,j\in\mathbb{N}}E[|Z^{x_{0},g_{0}}_{j,j+n}|^{2}]\leq \tilde{C}:=
2(N+1)[\hat{M}^{2}+M^{2}]\frac{2(2d+c_{1}+2\hat{M}^{2})}{c_{2}\eta}
+N\hat{M}^{2},
$$ 
see \eqref{chat}. Now for all $j\in\mathbb{N}$, let us apply Proposition \ref{heart} with $l:=j$,
$Q_{j}=(x_{0},g_{0})$ and
$Q_{j}':=Z^{X_{0},G_{0}}_{0,j}$ and with $\max\{\tilde{C},\hat{C}\}$ instead
of $\hat{C}$ (recall Remark \ref{poor}).
Now invoking
Lemma \ref{pixi} implies the statement.
\hfill $\Box$	
\end{proof}

\begin{proof}{Proof of Theorem \ref{main}.}
Lemma \ref{l2} and \eqref{cillag} imply that the sequence $\phi(X_{n},G_{n})$ is uniformly integrable. 
$\mathrm{Law}(X_{n},G_{n})$ tends to some $\mu_{*}$ in total variation, by Claim \ref{caim}. 
Let $(X_{*},G_{*})$ have law $\mu_{*}$. Necessarily,
$\mathrm{Law}(\phi(X_{n},G_{n}))$ also tends to $\mathrm{Law}(\phi(X_{*},G_{*}))$ in total variation,
a fortiori, in law. It follows that 
\begin{equation}\label{elle}
E[\phi(X_{n},G_{n})]\to \int_{\mathbb{R}^{(N+1)d}}\phi(u)\mu_{*}(du),\ n\to\infty.{}	
\end{equation}

Define $W_{n}:=\phi(X_{n},G_{n})-E[\phi(X_{n},G_{n})]$, $n\in\mathbb{N}$. This is a zero-mean strongly mixing 
sequence by obvious inequality $\alpha_{W}\leq \alpha_{(X,G)}$ and by Claim \ref{mixxi}. The sequence $W_{n}$, $n\in\mathbb{N}$
is also bounded in $L^{2}$ by \eqref{cillag} and by Lemma \ref{l2}.
 
Hence Lemma \ref{suf} and Theorem \ref{wlln} imply that 
\begin{equation}\label{mok}
\frac{\sum_{j=1}^{n}\left(\phi(X_{j},G_{j})-E[\phi(X_{j},G_{j})]\right)}{n}\to 0,\ n\to\infty,
\end{equation}
in $L^{1}$. Clearly, \eqref{elle} implies
also 
\begin{equation}\label{mak}
\frac{\sum_{j=1}^{n}E[\phi(X_{j},G_{j})]}{n}\to \int_{\mathbb{R}^{(N+1)d}}\phi(u)\mu_{*}(du),\ n\to\infty.
\end{equation}
Now \eqref{lln} follows trivially from \eqref{mok} and \eqref{mak}. 
\hfill $\Box$
\end{proof}

\section{On strongly mixing processes}\label{alpha}

For sigma-algebras $\mathcal{G},\mathcal{H}\subset\mathcal{F}$, define
$$
\alpha(\mathcal{G},\mathcal{H}):=\sup_{A\in\mathcal{G},B\in\mathcal{H}}|P(A\cap B)-P(A)P(B)|.
$$

For a (say, $\mathbb{R}^{m}$-valued) stochastic process $W_{j}$, $j\in\mathbb{N}$ we
define its $\alpha$-mixing coefficients as
$$
\alpha_{W}(n):=\sup_{j\in\mathbb{N}}\alpha(\sigma(W_{l},0\leq l\leq j),\sigma(W_{l},l\geq j+n)),\ n\in\mathbb{N},
$$ 
where $\sigma(\ldots)$ refers to the sigma-algebra generated by the respective random variables.
We call the process $W$ \emph{strongly mixing} if $\alpha_{W}(n)\to 0$, $n\to\infty$.

The following result appeared in \cite{hansen} for the first time. It was restated
and reproved in Theorem A.3 of \cite{attila}.
\begin{proposition}\label{wlln}
Let $W_{n}$, $n\in\mathbb{N}$ be a real-valued process with $\alpha_{W}(n)\to 0$, 
$n\to\infty$ such that 
\begin{equation}\label{ucint}
\lim_{V\to\infty}\sup_{n\geq 1}\frac{1}{n}\sum_{j=1}^{n}E[|W_{j}|1_{\{|W_{j}|\geq V\}}]= 0.	
\end{equation}
Then
\begin{equation}\label{connie}
E\left[\left|\frac{\sum_{j=1}^{n}[W_{j}-E[W_{j}]]}{n}\right|\right]\to 0,\ n\to\infty.	
\end{equation}
\end{proposition}

The following fact was mentioned right after Theorem A.3 in \cite{attila}.

\begin{lemma}\label{suf}
If $C_{W}:=\sup_{n\in\mathbb{N}}E[|W_{n}|^{2}]<\infty$ then \eqref{ucint} is satisfied.	
\end{lemma}
\begin{proof}{Proof.}
Indeed, the Cauchy and Markov inequalities imply that
$$
E[|W_{j}|1_{\{|W_{j}|\geq V\}}]\leq E[|W_{j}|^{2}]^{1/2}P(|W_{j}|\geq V)^{1/2}\leq C_{W}^{1/2}\left(\frac{C_{W}}{V^{2}}\right)^{1/2},
$$	
hence the supremum in \eqref{ucint} is dominated by $C_{W}/V$.\hfill $\Box$	
\end{proof}

The next lemma is a variant of Lemma 1.2 from \cite{attila}, adapted to the present setting. Remembering \eqref{iterate},
define the sigma-algebras $\mathcal{F}_{0,j}:=\sigma(Z^{X_{0},G_{0}}_{0,i},\ 0\leq i\leq j)$ and
$\mathcal{F}_{j,\infty}:=\sigma(Z^{X_{0},G_{0}}_{0,i},\ i\geq j)$. For any event $A\in\mathcal{F}$, $\bar{A}$ denotes
its complement in $\Omega$.

\begin{lemma}\label{pixi} Let $x_{0}\in\mathbb{R}^{d},g_{0}\in\mathbb{R}^{Nd}$ be arbitrary but fixed.
We have
$$
\alpha_{(X,G)}(n)\leq 2\sup_{j\in\mathbb{N}}P\left(Z^{Z^{X_{0},G_{0}}_{0,j}}_{j,j+n}\neq Z^{x_{0},g_{0}}_{j,j+n}\right),\ n\geq 1.
$$
\end{lemma}
\begin{proof}{Proof.} Let $j\in\mathbb{N}$ and $n\geq 1$ be arbitrary. Let $A\in\mathcal{F}_{0,j}$
and $B\in\mathcal{F}_{j+n,\infty}$.
Then there are $\tilde{A}\in\mathcal{B}(\mathbb{R}^{d(j+1)})$, 
$\tilde{B}\in\mathcal{B}((\mathbb{R}^{d})^{\mathbb{N}})$ 
such that $A=\{(Z_{0,k}^{X_{0},G_{0}})_{0\leq k\leq j}\in \tilde{A}\}$ and $B=\{(Z_{0,k}^{X_{0},G_{0}})_{k\geq j+n}\in \tilde{B}\}$.

Define ${B}':=\{(Z^{x_{0},g_{0}}_{j+n,k})_{k\geq j+n}\in \tilde{B}\}$.	
Notice that
$$
|P(A\cap B)-P(A)P(B)|=|\mathrm{cov}(1_{A},1_{B})|\leq |\mathrm{cov}(1_{A},1_{{B}'})|+|\mathrm{cov}(1_{A},1_{B}-1_{{B}'})|,
$$
where the first summand is $0$ since $A$ is 
$\sigma(X_{0})\vee \sigma(U_{i},U_{i}',0\leq i\leq j)$-measurable and that sigma-algebra is independent of 
$\sigma(U_{i},U_{i}',i\geq j+1)$ (which contains the event ${B}'$).
To estimate the second summand, notice that on the event $R:=\{Z_{0,j+n}^{X_{0},G_{0}}=Z^{x_{0},g_{0}}_{j,j+n}\}$ we have $1_{B}=1_{{B}'}$ so
\begin{eqnarray*}
|\mathrm{cov}(1_{A},1_{B}-1_{{B}'})| &\leq& |\mathrm{cov}(1_{A},(1_{B}-1_{{B}'})1_{R})|+
|\mathrm{cov}(1_{A},(1_{B}-1_{{B}'})1_{\bar{R}})|\\
&\leq&  0+\mathrm{var}^{1/2}(1_{A})
\mathrm{var}^{1/2}(1_{\bar{R}\cap B}-1_{\bar{R}\cap {B}'}).
\end{eqnarray*}
Here the first variance is at most $1$, the second is $\leq \mathrm{var}^{1/2}(1_{\bar{R}\cap B})
+\mathrm{var}^{1/2}(1_{\bar{R}\cap {B}'})$ which is at most $$
P(\bar{R}\cap B)+P(\bar{R}\cap {B}')\leq 2P(\bar{R})=2P(Z^{X_{0},G_{0}}_{0,j+n}\neq Z^{x_{0},g_{0}}_{j,j+n}).
$$
The latter probability is equal to $$
P\left(Z^{Z^{X_{0},G_{0}}_{0,j}}_{j,j+n}\neq Z^{x_{0},g_{0}}_{j,j+n}\right).{}
$$
As $A,B$ were arbitrary, we conclude that 
$$
\alpha(\mathcal{F}_{0,j},\mathcal{F}_{j+n,\infty})\leq 2P\left(Z^{Z^{X_{0},G_{0}}_{0,j}}_{j,j+n}\neq Z^{x_{0},g_{0}}_{j,j+n}\right)
$$
holds for all $j$. Notice that, by construction, $(Z^{X_{0},G_{0}}_{0,n})_{n\in\mathbb{N}}$ has the same law as
$(X_{n},G_{n})_{n\in\mathbb{N}}$, so
$$
\alpha(\mathcal{F}_{0,j},\mathcal{F}_{j+n,\infty})=
\alpha(\sigma((X_{i},G_{i}),0\leq i\leq j),\sigma((X_{i},G_{i}),i\geq j+n))
$$
and the statement follows. \hfill $\Box$
\end{proof}

\end{document}